\documentclass[11pt,letterpaper]{amsart}
\usepackage[utf8]{inputenc}
\usepackage[russian,english]{babel}
\usepackage{xcolor}
\usepackage{graphicx}
\usepackage{geometry}
\usepackage{amsmath,amsfonts,amssymb, amsthm, url}

\newcommand{\dd}{\mathcal{D}}
\newcommand{\pp}{\mathcal{P}}

\newcommand{\ff}{\mathcal{F}}
\newcommand{\s}{\mathcal{S}}
\newcommand{\g}{\mathcal{G}}
\newcommand{\aaa}{\mathcal{A}}
\newcommand{\bb}{\mathcal{B}}

\newcommand{\cc}{\mathcal{C}}
\newcommand{\uu}{\mathcal{U}}
\newcommand{\m}{\mathcal}

\newcommand{\T}{\mathcal{T}}
\newcommand{\W}{\mathcal{W}}

\newtheorem{thm}{Theorem}

\newtheorem{lem}[thm]{Lemma}

\newtheorem{conj}{Conjecture}
\newtheorem{obs}[thm]{Observation}
\newcommand\stir[2]{\genfrac{\{}{\}}{0pt}{0}{#1}{#2}}
\date{}
\title{Erd\H os--Ko--Rado type results for partitions via spread approximations}
\author{Andrey Kupavskii}
\address{Moscow Institute of Physics and Technology, St. Petersburg State University, Innopolis University, Russia; Email: {\tt kupavskii@ya.ru}}

\begin{document}

\maketitle
\begin{abstract}
  In this paper, we address several Erd\H os--Ko--Rado type questions for families of partitions. Two partitions of $[n]$ are {\it $t$-intersecting} if they share at least $t$ parts, and are {\it partially $t$-intersecting} if some of their parts intersect in at least $t$ elements. The question of what is the largest family of pairwise $t$-intersecting partitions was studied for several classes of partitions:   Peter Erd\H os and Sz\'ekely studied partitions of $[n]$ into $\ell$ parts of unrestricted size; Ku and Renshaw studied unrestricted partitions of $[n]$; Meagher and Moura, and then Godsil and Meagher studied partitions into $\ell$ parts of equal size. We improve and generalize the results proved by these authors.

  Meagher and Moura, following the work of Erd\H os and Sz\'ekely, introduced the notion of partially $t$-intersecting partitions, and conjectured, what should be the largest partially $t$-intersecting family of partitions into $\ell$ parts of equal size $k$. The main result of this paper is the proof of their conjecture for all $t, k$, provided $\ell$ is sufficiently large.

  All our results are applications of the spread approximation technique, introduced by Zakharov and the author. In order to use it, we need to refine some of the theorems from the original paper. As a byproduct, this makes the present paper a self-contained presentation of the spread approximation technique for $t$-intersecting problems.
\end{abstract}

\section{Introduction}
For a positive integer $n$, we use notation $[n] = \{1,\ldots, n\}$ and, more generally, $[a,b] = \{a,a+1,\ldots, b\}$. For a set $X$, the notation $2^X$ and ${X\choose k}$ stand for the collection of all subsets and all $k$-element subsets of the set $X$, respectively. A {\it family} is a collection of subsets of $X$ for some $X$.

This paper  deals with intersection theorems, which constitute a large and important class of results in extremal combinatorics. A family of sets is {\it intersecting} if any two sets in the family have non-empty intersection, and is {\it $t$-intersecting} if any two sets intersect in at least $t$ elements. In their seminal paper, Erd\H os, Ko and Rado \cite{EKR} determined the largest size of an intersecting family in $2^{[n]}$ and ${[n]\choose k}$ for all $n,k$. Later, after a series of  papers by Frankl \cite{F1}, Wilson \cite{Wil}, and Frankl and F\"uredi \cite{FF3}, Ahlswede and Khachatrian \cite{AK} confirmed Frankl's conjecture on the size of the largest $t$-intersecting family in ${[n]\choose k}$ for any $n,k,t$.

Intersecting questions were investigated for many other structures. See a great survey paper by Ellis \cite{Ell} on the subject. The first result in this direction was due to Deza and Frankl \cite{DF}, who determined the largest size of an intersecting family of permutations. We call two permutations $\sigma,\pi$ {\it intersecting} if for some element $x$ we have $\sigma(x) = \pi(x)$. Deza and Frankl also found the largest family of $2$- and $3$-intersecting permutations when $n$ and $n-1$ is a prime number, respectively. They also made a conjecture concerning the size of the largest $t$-intersecting family of permutations. The progress on this problem was rather difficult. After a series of papers, Ellis, Friedgut and Pilpel \cite{EFP} managed to solve it for any constant $t$, provided $n>n_0(t)$. Recently, the problem was resolved for any  $n$ and $t$ that satisfy $n>C t\log^2t$ by Zakharov and the author \cite{KuZa}. This was later improved to $n>Ct$ by Keller, Lifshitz, Minzer and Sheinfeld \cite{KLMS} and then to $n>(1+\epsilon)t$ for any $\epsilon>0$ by Kupavskii \cite{Kupper}. Early approaches to this question were algebraic, based on Hoffman-Delsarte type bounds and representation theory. The approach of \cite{EFP} combines junta approximations, coming from Boolean Analysis, with representation theory. Zakharov and the author introduced a combinatorial technique of spread approximations that is based on the breakthrough in the Erd\H os--Rado sunflower problem due to Alweiss, Lovett, Wu and Zhang \cite{Alw}. The approach of \cite{KLMS} is based on hypercontractivity, an important tool from Analysis.

Another class of EKR-type problems with algebraic flavour deals with different classes of partitions. Until now, it was approached either using the Delta-system method or algebraic tools, based on the Delsarte--Hoffman bounds for suitable graphs. In this paper, we show that the technique of spread approximations (an essentially combinatorial approach) allows to make significant progress in these questions.

Consider two partitions $P = (P_1,\ldots, P_{\ell_1})$ and $Q = (Q_1,\ldots, Q_{\ell_2})$ of $[n]$. We say that $P$ and $Q$ {\it $t$-intersect} if they share at least $t$ parts, that is, $P_{i_s} = Q_{j_s}$ for two sets of $t$ indices $\{i_1,\ldots, i_t\}$ and $\{j_1,\ldots, j_t\}$. We say that $P$ and $Q$ {\it partially $t$-intersect} if there are parts $P_i, Q_j$, such that $|P_i\cap Q_j|\ge t$. A family of partitions is {\it (partially) $t$-intersecting} if any two partitions from the family {\it (partially) $t$-intersect}. Note that two partitions may be $t$-intersecting, but not partially $t$-intersecting if each of their common parts has size strictly smaller than $t$.

\subsection{General partitions}
Let $\bb_n$ be the family of all partitions of $[n]$. Recall that $|\bb_n|$ is the $n$-th Bell number $B_n$ (for general reference, see \cite{bellwiki}). A natural example of $t$-intersecting family in $\bb_n$ is the family of all partitions that have $\{1\},\{2\},\ldots, \{t\}$ as parts. The size of this family is $B_{n-t}$.   Ku and Renshaw \cite{KuRe} proved the following result.
\begin{thm}[Ku and Renshaw \cite{KuRe}]\label{kure} Let $n\ge n_0(t)$ and assume that $\ff\subset \bb_n$ is $t$-intersecting. Then $|\ff|\le B_{n-t}$, with equality only possible if $\ff$ is a family of all partitions with $t$ fixed singletons.
\end{thm}
In their proof, they require $t\le c\log n$.  The theorem below gives a much better dependence between the parameters.

\begin{thm}\label{thm6} The conclusion of Theorem~\ref{kure} is valid for $n\ge Ct\log^2t$ with some absolute constant $C$.
\end{thm}

\subsection{Partitions with $k$ parts}
Let $\pp_n^\ell$ be the family of all partitions of $[n]$ into $\ell$ parts. Recall that $|\pp_n^\ell| = \stir{n}{\ell}$, where the expression on the right is the Stirling number of the second kind (for general reference, see \cite{stirwiki}). Peter Erd\H os and L\'aszlo Sz\'ekely \cite{ErSz} proved the following result.
\begin{thm}[Erd\H os and Sz\'ekely \cite{ErSz}]\label{thmersz} Let $n\ge n_0(\ell)$ be large enough. Assume that $\ff\subset \pp_n^\ell$ is $t$-intersecting. Then $|\ff|\le \stir{n-t}{\ell-t}$.
\end{thm}

Here, we improve their result as follows.
\begin{thm}\label{thmp1} Let $n,\ell,t$ be integers such that $t\le \ell-2$, $n\ge 2\ell\log_2 n$ and $n\ge 48$. Assume that $\ff\subset \pp_n^\ell$ is $t$-intersecting. Then $|\ff|\le \stir{n-t}{\ell-t}$. Moreover, if $\ff$ is not contained in a family of all partitions with $t$ fixed singletons, then $|\ff|\le \frac 12\stir{n-t}{\ell-t}$.
\end{thm}
Note that the question is trivial for $t=\ell-1$, since if two $k$-partitions have $\ell-1$ common parts then the last part is also common, so they coincide.

\subsection{Partitions with fixed profile}
Let $P = (k_1,\ldots, k_\ell)$ be a non-decreasing sequence of positive integers, called a {\it profile}, and put $n = \sum_{i=1}^\ell k_i$. Consider the family of partitions $\uu_P$  of $n$ into $\ell$ parts, where the $i$-th part has size $k_i$ ({\it partitions with profile $P$}). A {\it canonical $t$-intersecting family} of such partitions is a family $\aaa_{P,t}^X$, defined by a $t$-tuple $X$ of disjoint sets $X_1,\ldots, X_t$, where $|X_i| = k_i$. It consists of all partitions that contain each $X_i$ as one of its parts.

A {\it $(k,\ell)$-partition} is a particular type of profiled partitions, when $k_1 = \ldots = k_\ell =: k$. In other words, it is a partition of $[k\ell]$ into exactly $\ell$ blocks, each of size $k$. Let $\uu_{k,\ell}$ stand for the family of all $(k,\ell)$-partitions, and put $u_{k,\ell}:= |\uu_{k,\ell}|.$

Meagher and Moura \cite{MeMo} proved the following theorems.

\begin{thm}[\cite{MeMo}] Fix positive integers $k,\ell$. Let $\ff\subset\uu_{k,\ell}$ be an intersecting family of $(k,\ell)$-partitions. Then $|\ff|\le u_{k,\ell-1}$, and the equality is only possible for canonical intersecting families.
\end{thm}

\begin{thm}[\cite{MeMo}] Fix positive integers $k,\ell,t$. Suppose that either $(k\ge k_0(\ell,t))$ or $(k\ge t+2$ and $\ell\ge \ell_0(k,t))$. Let $\ff\subset\uu_{k,\ell}$ be a $t$-intersecting family of $(k,\ell)$-partitions. Then $|\ff|\le u_{k,\ell-t}$, and the equality is only possible for canonical $t$-intersecting families.
\end{thm}

In the theorem below, we extend their results to a large class of profiled partitions, significantly improve the dependence between the parameters, and give a strong stability result.

\begin{thm}\label{thmu1} Let $t,\ell$ be positive integers and $P = (k_1,\ldots, k_\ell)$ be a profile. Assume that the following holds:  $\ell\ge 1000$; $t\le \ell/2$; $k_{t+1}\ge 3$.   If $\ff\subset \uu_P$ is $t$-intersecting then $|\ff|\le |\aaa_{P,t}^X|$. Moreover, if $\ff$ is not contained in $\aaa_{P,t}^X$ for some $X$ then $|\aaa|\le \frac 12 |\aaa_{P,t}^X|$.
\end{thm}
We did not try to optimize constant $1000$ in the statement. The point is that the bound is reasonable.

\subsection{Partially $t$-intersecting partitions}
In the notation of the previous subsection, let $P = (k_1,\ldots, k_\ell)$ be a profile that additionally satisfies $k:=k_\ell \ge t$. Put $n = \sum_i k_i$.  What is the largest partially $t$-intersecting subfamily of $\uu_P$? A natural candidate is a {\it canonically partially $t$-intersecting family} $\cc^T_P\subset \uu_P$, where $T$ is a set of size $t$, and $\cc^T_P$ consists of all partitions from $\uu_P$ that fully contain $T$ in one of its parts.
Similarly, for the case of $(k,\ell)$-partitions,   a {\it canonically partially $t$-intersecting partition} is a family $\cc_{k,\ell}^T\subset \uu_{k,\ell}$ that consists of all $(k,\ell)$-partitions with a part that contains a fixed set $T$, $|T| = t$.  Meagher and Moura \cite{MeMo} conjectured the following.
\begin{conj}[Meagher and Moura, \cite{MeMo}]
  Let $\ell,k,t$ be integers and $\ff\subset \uu_{k,\ell}$ be a partially $t$-intersecting family. Then, for a $t$-element $T\subset [k\ell]$,
  $$|\ff|\le |\cc_{k,\ell}^T|.$$
\end{conj}
Note that any two partitions are partially $1$-intersecting, and thus $\uu_{k,\ell}$ itself is a partially $1$-intersecting family. The Meagher--Moura conjecture was proved by Godsil and Meagher \cite{GodMe} for $k=t=2$ and all $\ell$. Meagher, Shirazi and Stevens \cite{MSS} proved it for $t=2$ and any $k$, provided $\ell$ is sufficiently large. The approach in the first paper is algebraic, while the second paper uses Delta-systems.

We prove the following theorem, which is the main result of this paper.
\begin{thm}\label{thm1}
  The Meagher--Moura conjecture is true for any $k,t$, provided $\ell$ is sufficiently large. Moreover, the only extremal examples have the form $\cc_{k,\ell}^T$ for some $T$.
\end{thm}

This could be generalized to a much larger class of profiled partitions.

\begin{thm}\label{thm2}
  For any $k,t$ there exists $\ell$ such that for any $P$ as above every largest partially $t$-intersecting family of $\uu_P$ is of the form $\cc^T_P$ for some $T$.
\end{thm}
The proof that we present for Theorem~\ref{thm1} goes almost verbatim for Theorem~\ref{thm2}, except for the more complicated notation and some slightly different estimates. Thus, we decided to omit it. We also note that in both settings the proofs give a stability result that, informally speaking, tells that any family that is within constant factor of the size of the extremal family must be very close to one of the extremal families.

\subsection{The structure of the remainder of the paper}
In order to prove the theorems, we use the spread approximation method. We need to make some improvements in the main ingredients of the method, relaxing the conditions that we impose on the families, as compared to \cite{KuZa}. For this reason, we present the proofs of the results from \cite{KuZa}, taking into account the requirements needed for this paper. The proof of the first theorem stays essentially the same, while the second theorem is simplified and has an improved bound as compared to \cite{KuZa}. This makes this paper a self-contained presentation of the version of the spread approximation technique for $t$-intersecting problems. We also note that, while working in the spread approximation framework gives us a huge lever, each of the problems poses its own difficulties that we need to overcome.

In the next section we give an outline of the proofs of our theorems. In Section~\ref{sec3}, we present the parts of the spread approximation method that we use in this paper. In Section~\ref{sec4}, we give the proofs of our main results. We note that in Section~\ref{sec4} we prove an almost sharp lower bound on the ratio of consecutive Bell numbers, as well as the lower bound for the number of partitions with parts of sizes at least $2$, which may be of some independent interest.

\section{Sketches of the proofs}\label{sec2}
As we have already mentioned, the proofs of our Theorems are based on the spread approximation method, introduced by Zakharov and the author \cite{KuZa}. The steps of the proof are summarized below.

\begin{enumerate}
  \item Reformulate the problem in terms of families of sets. The problems we deal with are for families of partitions, and the method of spread approximation is designed for families of sets. Thus, we need to give an adequate set interpretation to the problems. Theorems~\ref{thm6},~\ref{thmp1}, and~\ref{thmu1} are based on a more straightforward interpretation: each possible part in a partition is treated as a singleton, and thus partitions become sets of size at most $n$ or exactly $\ell$, depending on the question. Theorems~\ref{thm1} and~\ref{thm2} require a more complicated interpretation. There, each pair of elements from the original ground set becomes a singleton in the new ground set, and each partition is turned into a set of all pairs of singletons that lie in the same part.
  \item Prove that the set family, corresponding to the ambient family of partitions, is sufficiently spread (this is a certain quasirandomness notion that is crucial for the spread approximations method and that we introduce in the next section). This is again more or less complicated for different settings. For example, in the proof of Theorem~\ref{thm6}, in order to lower bound spreadness, we need to lower bound the ratio $B(n)/B(n-1)$.
  \item Take the extremal intersecting family $\ff$ (depending on the setting) and apply the spread approximation theorem. Get a lower-uniformity family $\s$, which encodes a family of ``partial'' partitions and that covers most of our family, as well as a small remainder $\ff'\subset \ff$ that is not covered by $\s$. The approximation is given by Theorem~\ref{thmkz1}.
  \item Show that this lower-uniformity family $\s$ is trivial for the extremal family: it consists of a single $t$-element set $T$ (corresponding to different ``partial'' partitions depending on the setting). This is done using Theorem~\ref{thmkz2}.
  \item Show that, for an extremal family $\ff$, the remainder $\ff'$ must be empty. This is more of an ad-hoc argument, which requires a significant amount of effort in some cases. In particular, for the case of $\bb(n)$, we need to lower bound the number of partitions with parts of sizes at least $2$. In the case of partially $t$-intersecting families, we need to lower bound the probability that two random partitions of the type as in Theorems~\ref{thm1},~\ref{thm2} are not partially $t$-intersecting. Such bounds allow us to say that what we ``lose'' because of having some partitions in $\ff'$ cannot be compensated by what we ``gain'' by adding sets from $\ff'$ (keeping in mind that we have upper bounds on the size of $\ff'$).
\end{enumerate}
For Theorems~\ref{thmp1} and~\ref{thmu1}, steps 3 and 5 are omitted, since these families are sufficiently spread so that we can apply Theorem~\ref{thmkz2} directly to $\ff$ instead of $\s$. Thanks to that, we get the strongest stability results in these cases. We note that Theorem~\ref{thmkz2} alone can be seen as a strengthening of one of the important parts of the Delta-system method.

In the third and fourth step we need to do some refinement of the method, proposed by Kupavskii and Zakharov. The fourth step is simplified and improved as compared to \cite{KuZa}.

\section{Spread approximations}\label{sec3}
For a set $X\subset [n]$ and families $\g,\m S\subset 2^{[n]}$ we use the following standard notation:
  \begin{align*}\g(X):=&\{A\setminus X: A\in \g, X \subset A\},\\
  \g[X]:=&\{A: A\in \g, X \subset A\},\\
  \g[\m S]:=&\bigcup_{X\in \m S} \g[X],\\
  \g(\bar X):=&\{A: A \in \g, A\cap X = \emptyset\}.\end{align*}
  We think of $\g(X),\g(\bar X)$ as of subfamilies of $2^{[n]\setminus X}.$
Let $r>1$ be some real number. We say that a family $\ff$ is {\it $r$-spread} if, for any set $X$, $|\ff(X)|\le r^{-|X|}|\ff|$. We denote by $\|\ff\|$ the average size of a set sampled from $\ff$: $\|\ff\|:=\frac 1{|\ff|}\sum_{S\in \ff}|S|$. Note that $\|\ff\|$ is at most the size of the largest set in $\ff$.
We say that $W$ is a {\it $p$-random subset} of $[n]$ if each element of $[n]$ is included in $W$ with probability $p$ and independently of others.

The following statement is a variant due to Tao \cite{Tao} of the breakthrough result that was proved by Alweiss, Lowett, Wu and Zhang \cite{Alw}.

\begin{thm}[\cite{Alw}, a sharpening due to \cite{Tao}]\label{thmtao}
  Suppose that $\delta\in (0,1)$, and that $m$ is a positive integer such that $\delta m< 1$. If for some $n,r\ge 1$ a family $\ff\subset 2^{[n]}$ is $r$-spread and $W$ is an $(m\delta)$-random subset of $[n]$, then $$\Pr[\exists F\in \ff\ :\ F\subset W]\ge 1-\Big(\frac 5{\log_2(r\delta)} \Big)^m\|\ff\|.$$
\end{thm}
Recall that an {\it $\ell$-sunflower} is a collection of $\ell$ sets $F_1,\ldots, F_\ell$ such that for any $i\ne j$ we have $F_i\cap F_j = \cap_{i=1}^\ell F_i$. (In particular, $\ell$ pairwise disjoint sets form an $\ell$-sunflower.) The set $\cap F_i$ is called the {\it core}. The theorem implies an important strengthening on the size of the family that guarantees the existence of an $\ell$-sunflower. Namely, it implies that any family $\ff$ of $k$-sets with \begin{equation}\label{eqalweiss}|\ff|> \big(C\ell\log_2 (k\ell)\big)^k\end{equation} contains an $\ell$-sunflower, where $C$ is an absolute constant and can be taken to be $2^{10}$.

To construct spread approximations, we will need the following easy observation.
\begin{obs}\label{obs2} Given $r>1$ and a family $\ff\subset 2^{[n]}$, let $X$ be an inclusion-maximal set that satisfies $|\ff(X)|\ge r^{-|X|} |\ff|$. Then $\ff(X)$ is $r$-spread  as a family in $2^{[n]\setminus X}$.
\end{obs}
\begin{proof}
Indeed, by maximality, for any $B\supsetneq X$ of size $b$ we have
$$|\ff(B)|\le r^{-b} |\ff|\le r^{-b+|X|} |\ff(X)|.$$
\end{proof}
To relate size and spreadness, we use the following observation.
\begin{obs}\label{obs3}
  If for some $\alpha>1$ and $\ff\subset {[n]\choose k}$ we have $|\ff|>\alpha^k$ then $\ff$ contains an $\alpha$-spread subfamily of the form $\ff(X)$ for some set $X$ of size strictly smaller than $k$.
\end{obs}
\begin{proof}
If $\ff$ is $\alpha$-spread then we are done. Otherwise, take the largest $X$ that violates the $\alpha$-spreadness of $\ff$. Clearly, $|X|<k$ and $\ff(X)$ satisfies the requirements of the observation.
\end{proof}
We note that this observation together with Theorem~\ref{thmtao} implies bound \eqref{eqalweiss}.

The next theorem allows to construct low-uniformity approximations for sufficiently spread families.
\begin{thm}\label{thmkz1}
  Let $n,k,t\ge1$ be some integers and $\aaa\subset 2^{[n]}$ be a family. Consider a family $\ff\subset \aaa\cap {[n]\choose \le k}$ that is  $t$-intersecting. Let $q, r, r_0 \ge 1$ satisfy the following: $r>2^{12}\log_2(2k)$, $r\ge 2q$ and $r_0>r$.  Assume that $\aaa$  is $r_0$-spread.

Then there exists  a $t$-intersecting family $\s$ of sets of size at most $q$ ({\emph a spread approximation of $\ff$}) and a `remainder' $\ff'\subset \ff$ such that
\begin{itemize}
    \item[(i)] We have $\ff\setminus \ff'\subset \aaa[\s]$;
    \item[(ii)] for any $B\in \s$ there is a family $\ff_B\subset \ff$ such that $\ff_B(B)$ is $r$-spread;
    \item[(iii)] $|\ff'|\le (r_0/r)^{-q-1}|\aaa|$.
  \end{itemize}
\end{thm}
The crucial difference with \cite[Theorem 11]{KuZa} is that we only require $r$-spreadness from $\aaa$, instead of {\it $(r,t)$-spreadness}\footnote{$(r,t)$-spreadness means that any subfamily of the form $\aaa(T)$ with $|T|\le t$ must be $r$-spread)}. This is  crucial in our application to Theorem~\ref{thm1}, because the set interpretation of $\uu_{k,\ell}$ is an $r$-spread, but not an $(r,t)$-spread, family. The proof of this theorem is essentially the same as in \cite{KuZa}. We present it here for completeness.
\begin{proof}
The theorem is obtained using the following procedure.  For $i=1,2,\ldots $ with $\ff^1:=\ff$ we do the following steps.
\begin{enumerate}
    \item Find an inclusion-maximal set $S_i$ such that  $|\ff^i(S_i)|\ge  r^{-|S_i|}  |\ff^i|$;
    \item If $|S_i|> q$ or $\ff^i = \emptyset$ then stop. Otherwise, put $\ff^{i+1}:=\ff^i\setminus \ff^i[S_i]$.
\end{enumerate}

The family $\ff^i(S_i)$ is $r$-spread by Observation~\ref{obs2}.

Let $m$ be the step of the procedure for $\ff$ at which we stop. Put  $\s:=\{S_1,\ldots, S_{m-1}\}$. Clearly, $|S_i|\le q$ for each $i\in [m-1]$. The family $\ff_{B}$ promised in (ii) is defined to be $\ff^i[S_i]$  for $B=S_i$. Next, note that if $\ff^m$ is non-empty, then $$|\ff^m|\le r^{|S_m|} |\ff^{m}(S_m)|\le r^{|S_m|} |\aaa(S_m)|\le (r/r_0)^{|S_m|} |\aaa|,$$
where in the last inequality we used the $r_0$-spreadness of $\aaa$.
We put $\ff':=\ff^m$. Since either $|S_m|>q$ or $\ff' = \emptyset $, we have $|\ff'|\le (r_0/r)^{-q-1}|\aaa|$.

The only thing left to verify is the $t$-intersection property.
Take any (not necessarily distinct) $S_i,S_j\in \mathcal S$  and assume that $|S_i\cap S_j|<t$. Recall that  $\g_i:=\ff^i(S_i)$ and $\g_j:=\ff^j(S_j)$ are both $r$-spread.
  $$|\g_j(\bar S_i)| \ge |\g_j|-\sum_{x\in S_i\setminus S_j} \big|\g_j[\{x\}]\big|\ge \Big(1-\frac {|S_i|}{r}\Big) |\g_j|\ge \frac 12|\g_j|.$$
  In the last inequality we used that $|S_i|\le q$ and that $r\ge 2q.$ The same is valid for $\g_i(\bar S_j)$. Note that both $\g_j':=\g_j(\bar S_i)$ and $\g_i':=\g_i(\bar S_j)$ are subfamilies of $2^{[n]\setminus (S_i\cup S_j)}.$ Because of the last displayed inequality and the trivial inclusion $\g_j'(Y)\subset \g_j(Y)$, valid for any $Y$, both $\g_i'$ and $\g_j'$ are $\frac r2$-spread, where  $\frac r2 > 2^{11}\log_2(2k)$. Indeed, we have
  $$2|\g_j'|\ge |\g_j|\ge r^{|X|}|\g_j(X)|\ge r^{|X|}|\g_j'(X)|.$$


  We are about to apply Theorem~\ref{thmtao}. Let us put $m= \log_2(2k)$, $\delta = (2\log_2(2k))^{-1}$, and let $r/2$ play the role of $r$. Note that $m\delta = \frac 12$ and $\frac r2\delta > 2^{10}$ by our choice of $r$.  Theorem~\ref{thmtao} implies that a $\frac{1}{2}$-random subset $W$ of $[n]\setminus (S_i\cup S_j)$ contains a set from $\g_j'$ with probability strictly bigger than
  $$1-\Big(\frac 5{\log_2 2^{10}}\Big)^{\log_2 2k} k = 1-2^{-\log_2 2k} k = \frac 12.$$

  Consider a random partition of $[n]\setminus (S_i\cup S_j)$ into $2$ parts $U_i,U_j$, including each element with probability $1/2$ in each of the parts. Then both $U_\ell$, $\ell\in \{i,j\}$, are distributed as $W$ above. Thus, the probability that there is $F_\ell \in \g_\ell'$ such that $F_\ell\subset U_\ell$ is strictly bigger than $\frac 12$. Using the union bound, we conclude that, with positive probability, it holds that there are such $F_\ell$ with $F_\ell\subset U_\ell$ for each  $\ell \in\{i,j\}$. Fix such choices of $U_\ell$ and $F_\ell$, $\ell \in \{i,j\}$. Then, on the one hand, both $F_i\cup S_i$ and $F_j\cup S_j$ belong to $\ff$ and, on the other hand, $|(F_i\cup S_i)\cap (F_j\cup S_j)| = |S_i\cap S_j|<t$, a contradiction with $\ff$ being $t$-intersecting.
\end{proof}

An important second step is to show that the approximation family $\s$ is trivial for an extremal family. In  \cite[Theorem 12]{KuZa}, the authors worked with $(r,t)$-spread families, which we cannot afford in the proof of Theorem~\ref{thm1} for the family $\uu_{k,\ell}$, and thus have to find ways around it. Let us say that a family $\aaa\subset 2^{[n]}$ is {\it weakly $(r,t)$-spread} if there exists a set $T\subset [n]$ of size $t$ such that for any nonnegative integer $s$ and a set $U\subset [n]$ of size $t+s$ we have $|\aaa(U)|\le r^{-s}|\aaa(T)|$. Informally speaking, this is akin to $r$-spreadness for the family $\aaa(T)$, where $T$ of size $t$ is chosen so that $|\aaa(T)|$ is maximal.

 Recall that a $t$-intersecting family $\s$ is {\it non-trivial } if $|\cap_{F\in \s} F|<t.$

\begin{thm}\label{thmkz2}
Let $\varepsilon\in (0,1]$, $n,r,q, t \ge 1$  be such that $\varepsilon r\ge 24 q$.
Let $\aaa \subset 2^{[n]}$ be a weakly $(r, t)$-spread family and let $\s \subset {[n] \choose \le q}$ be a non-trivial $t$-intersecting family.
Then there exists a $t$-element set $T$ such that $|\aaa[\s]| \le \varepsilon |\aaa[T]|$.
\end{thm}
The proof again follows closely the proof of  \cite[Theorem 12]{KuZa}, with a few changes. Again, we present it here  in full for completeness. 
 In what follows, assume that $T$ is a set of size $t$ that maximizes $|\aaa(T)|$.
For the proof, we will need the following simple observation.

\begin{obs}\label{obs22}
For any positive integers $n,p$, a family $\aaa\subset 2^{[n]}$ and a $t$-intersecting family $\s \subset {[n] \choose \le p}$ there exists a $t$-intersecting family $\T \subset {[n] \choose \le p}$ such that $\aaa[\s] \subset \aaa[\T]$ and for any $T \in \T$ and any proper subset $X \subsetneq T$ there exists $T' \in \T$ such that $|X \cap T'| < t$.
\end{obs}
One natural way to choose such $\T$ is to repeatedly replace sets in $\s$ by their proper subsets while preserving the $t$-intersecting property.

In terms of Theorem~\ref{thmkz2}, let us iteratively define the following series of families.
\begin{enumerate}
    \item Let $\T_0$ be a family given by Observation~\ref{obs22} when applied to $\aaa$ and $\s$ with $p= q$.
    \item For $i = 0, \ldots, q-t$ we put $\W_i = \T_i \cap {[n] \choose q-i}$ and let $\T_{i+1}$ be the family given by Observation~\ref{obs22} when applied to the families $\aaa$ (playing the role of $\aaa$) and  $\T_{i}\setminus \W_{i}$ playing the role of $\mathcal S$ with $p = q-i-1$.
\end{enumerate}
Remark that $\T_i$ is $t$-intersecting for each $i=0,\ldots, q-t$ by definition. We summarize the properties of these series of families in the following lemma.

\begin{lem}\label{lemkeyred} The following properties hold for each $i = 0,\ldots, q-t$. 
\begin{itemize}
  \item[(i)] All sets  in $\T_i$ have size at most $q-i$.
  \item[(ii)] We have $\aaa[\T_{i-1}]\subset \aaa[\T_{i}]\cup \aaa[\W_{i-1}]$.
  \item[(iii)] The family $\T_i$ does not have a subfamily $\mathcal Y$ and a set $X$ such that $\mathcal Y(X)$ satisfies $|\mathcal Y(X)|>1$ and is $>(q-i-t+1)$-spread.
  \item[(iv)] We have $|\W_i|\le (6(q-i))^{q-i-t}$.
  \item[(v)] If $\T_i$ consists of a single $t$-element set $X$ and this is not the case for $\T_{i-1}$ then $|\aaa[\T_{i-1}\setminus \W_{i-1}]|\le \frac{q}r |\aaa[T]|$.
\end{itemize}
\end{lem}
The lemma is similar \cite[Lemma 14]{KuZa}, with two major differences. First, we replace the use of sunflowers by the use of spread families, which improves the bound in (iv). Second, in (v), where we replace a concrete $X$ with a ``universal'' $T$ on the right-hand side.
\begin{proof}
(i) This easily follows by induction on $i$ from the fact that all sets in $\s$ have size at most $q$ and the definition of $\T_i$.

(ii) We have $\aaa[\T_{i-1}] = \aaa[\T_{i-1} \setminus \W_{i-1}] \cup \aaa[\W_{i-1}]$ and, by the definition of $\T_{i}$, we have $\aaa[\T_i]\supset \aaa[\T_{i-1}\setminus \W_{i-1}]$.

(iii) Assume such $\mathcal Y$ and $X$ exist. 
Assume that a set $T'\in \T_i$ intersects $X$ in $t-j$ elements, $j>0$. Then $T'$ intersects each set in $\mathcal Y(X)$ in at least $j$ elements. For any set $Z$ disjoint with $X$ we have $|\mathcal Y(X\cup Z)|< (q-i-t+1)^{-|Z|}|\mathcal Y|$. There are, however, only ${T'\setminus X\choose j}\le {q-i-t+j\choose j}$ subsets of $T'\setminus X$ of size $j$, and thus the part of $\mathcal Y(X)$ that intersects $T'$ in at least $j$ elements has size strictly smaller than
$${q-i-t+j\choose j}(q-i-t+1)^{-j}|\mathcal Y|= (q-i-t+1)^{-j}\prod_{a=1}^j\frac{q-i-t+a}{a}|\mathcal Y|\le |\mathcal Y|.$$
This is a contradiction, and thus $j\le 0$. But then this contradicts the minimality of $\mathcal T_i$: we could have replaced $\mathcal T_i(X)$ with $X$.

(iv) This is trivial for $i = q-t$ since $\T_{q-t}$ contains at most $1$ set. In what follows, we assume that $i<q-t$. Take any set $Y\in \W_i$. Since $\T_i$ is $t$-intersecting, there is a $t$-element subset $X\subset Y$ such that $|\W_i| \le {q-i \choose t} |\W_i(X)|={q-i \choose q-i-t} |\W_i(X)|$. Next, $\W_i(X)$ is $(q-i-t)$-unform and does not contain a subfamily $\mathcal Y$ and a set $X$ such that  $\mathcal Y(X)$ is $>(q-i-t+1)$-spread. Using Observation~\ref{obs3}, we conclude that $|\W_i(X)|\le (q-i-t+1)^{q-i-t}$. 
\begin{align*}
|\W_i| \le& {q-i \choose q-i-t} (q-i-t+1)^{q-i-t}\\ 
\le& \Big(\frac{e(q-i)}{q-i-t}\Big)^{q-i-t}(2(q-i-t))^{q-i-t} \\
\le& 6(q-i)^{q-i-t} 
\end{align*}

(v) Let us assume that $\T_i = \{X\}$ for some $t$-element set $X$. Note that all sets in $\T_{i-1}$ have size at least $t+1$. Otherwise, if there is $X'\in \T_{i-1}$ of size $t$ then $X'$ is a proper subset of all other sets from $\T_{i-1},$ which contradicts the property of $\T_{i-1}$ guaranteed by Observation~\ref{obs22}. Thus, the sets in $\T':=\T_{i-1}\setminus \W_{i-1}$, if any, have size at least $t+1$ and all contain $X$. Recall that, for a family $\ff$, {\it the covering number} $\tau(\ff)$ is the size of the smallest cover: a set $Y$ such that $Y\cap F\ne \emptyset$ for each $F\in \ff.$ Assume that $\tau(\T'(X))> q$. Each set in $\W_{i-1}$ either contains $X$ or intersects every set from $\T'(X)$. In the latter case, it has size at least $\tau(\T'(X))$, which is impossible because each set in $W_{i-1}$ has size at most $q$. Thus, all sets from $\W_{i-1}$ contain $X$, implying that all sets from $\m T_{i-1}$ contain $X$, a contradiction. Therefore, $\tau(\T'(X))\le q$. Recall that $T, |T|=t$, is such that $|\aaa(T)|$ is maximal. If $\{x_1, \ldots, x_q\}$ is a covering of $\T'(X)$ then we have
$$
|\aaa[\T']| \le |\aaa[X \cup \{x_1\}]| + \ldots + |\aaa[X\cup \{x_q\}]| \le \frac qr |\aaa[T]|,
$$
where in the last inequality we used the definition of $T$ and the weak $(r,t)$-spreadness.
\end{proof}

\begin{proof}[Proof of Theorem~\ref{thmkz2}] 
Fix $i$ as in Lemma~\ref{lemkeyred} (v). Note that by (i) such a choice always exists. Let $T$ be a $t$-element set such that $|\aaa[T]|$ is maximal. By the weak $(r, t)$-spreadness, for any $j < i$ and any $W \in \W_j$ we have $|\aaa[W]| \le r^{-(q-j - t)} |\aaa[T]|$. By (iv) and the union bound, we get $|\aaa[\W_j]| \le r^{-(q-j-t)} (6q)^{(q-j-t)} |\aaa[T]|$. Using this and (v) we obtain
\begin{align*}
|\aaa[\s]|\overset{(ii)}{\le} |\aaa[\T_{i-1}]|+\sum_{j=0}^{i-1}|\aaa[\W_j]|\overset{(iv),(v)}{\le}& \Big(\frac q r +\sum_{j=1}^{\infty} r^{-j}(6q)^{j} \Big)|\aaa[T]|\\ \le  \ \ &\Big(\frac \epsilon 2 +\sum_{j=1}^{\infty} \big(\frac \epsilon4\big)^j \Big)|\aaa[T]|\le \varepsilon |\aaa[T]|,
\end{align*}
where in the third inequality we used the condition on $r$ and the bound on $C_0$.
\end{proof}

\section{Proofs}\label{sec4}

\subsection{General partitions. Proof of Theorem~\ref{thm6}}

\begin{lem}\label{lemb1}
  For any $n\ge 2$ we have $\frac{B_{n+1}}{B_{n}}\ge \frac{n}{2\log_e n}$.
\end{lem}
\begin{proof}
  We use the following remarkable explicit formula for $B_n$ (see \cite{Wilf}):
  \begin{equation}\label{eqb1} B_n = \frac 1e \sum_{s=0}^{\infty} \frac{s^n}{s!}.\end{equation}
  Let us compare the $(s-1)$-th and $s$-th terms in the summation. We have
  \begin{equation}\label{eqb2} \frac{\frac{(s-1)^n}{(s-1)!}}{\frac{s^n}{s!}} = s\Big(1-\frac 1{s}\Big)^n\le se^{-n/s}<e^{-1},\end{equation}
  provided $\frac n{s}\ge 1+\log_es$, that is, $n\ge s+s\log_es$. Let us put $s_0(n) = \big\lfloor\frac{n}{\log_e n}\big\rfloor$ and note that $s_0(n)$ satisfies the last inequality for $n\ge 16$. (It is easy to verify by direct substitution of a larger value $s=\frac{n}{\log_e n}$ and using that $e^e<16$.)

  Using \eqref{eqb2}, we conclude that the terms in \eqref{eqb1} grow faster than a geometric progression with base $e$ until at least $s_0$, and so, using the formula for the sum of a geometric progression, we get $\sum_{s=0}^{s_0-1} \frac{s^n}{s!}\le \frac 1{e-1}\frac{s_0^n}{s_0!}.$ We can thus bound $B_n$ as follows:
  $$B_n\le \frac{e}{e-1}\sum_{s=s_0}^{\infty}\frac{s^n}{s!}.$$
  (We note that, of course, with a bit more care and for relatively large $n$, the fraction in front should be essentially $1$.)
  Let us now bound the ratio.
  $$\frac{B_{n+1}}{B_n}\ge \frac{\sum_{s=s_0(n)}^{\infty}\frac{s^{n+1}}{s!}} { \frac{e}{e-1}\sum_{s=s_0(n)}^{\infty}\frac{s^n}{s!}}\ge s_0(n)\cdot\frac{ (e-1)}{e}\ge\frac{n}{\log_e n}\cdot \frac 45\cdot \frac{ (e-1)}{e}\ge \frac{n}{2\log_e n},$$
where the second to last inequality holds for $s_0(n)\ge 4$ by the definition of $s_0(n)$. This completes the proof for $n\ge 16$ and $\big\lfloor\frac{n}{\log_e n}\big\rfloor\ge 4$. The last inequality holds for $n\ge 16$ as well, and thus we are left to verify the statement for $n\le 15.$

We have $B_n\ge 2B_{n-1}$ since, for each partition $P$ of $[n-1]$, element $n$ can be either made a separate part or adjoined to one of the parts in $P$. Similarly, $B_n\ge 3B_{n-1}$ for $n\ge 5$ because for each partition $P$ but the single partition with $1$ part it can be extended by $n$ in at least $3$ ways and, moreover, there are partitions that can be extended in at least $4$ ways. We are left to note that $n/(2\log_e n)\le 2$ for $2\le n\le 8$ and $n/(2\log_e n)\le 3$ for $9\le n\le 15$.
\end{proof}

Given a partition $P  =(P_1,\ldots, P_\ell)\in \m B_n$, we denote by $\dd_n\langle P\rangle\subset \bb_n$ the family of all partitions from $\bb_n$ that do not contain any of $P_i$ as a part. We call this family {\it $P$-derangements}.
\begin{lem}\label{lemb2} For  any partition $P\in \bb_n$ we have $|\dd_n\langle P\rangle|\ge c' e^{-\frac 32\log_e^2n} B_n$.
\end{lem}
Following our approach with more care, the constant in front of $\log_e^2s$ can be improved to $\frac 12+\epsilon$, provided $n\ge n_0(\epsilon)$ is large enough.
\begin{proof}
  The first step, inspired by a compression operation from Ku and Renshaw \cite{KuRe}, is to show that $|\dd_n\langle P\rangle|$ is minimized for a partition that consists of $n$ singletons. To do so, for any given partition $P$ that has a part $P_i$ with $|P_i|\ge 2$, we introduce a new partition $P'$ that is the same as $P$ except that it replaces part $P_i$ with two parts $Q_i,Q'_i$, where these are non-empty disjoint sets such that $Q_i\sqcup Q'_i = P_i$. Take any partition $U\in\dd_n\langle P'\rangle\setminus \dd_n\langle P\rangle$. Partitions $P$ and $P'$ mostly coincide, and the only part that the former has and the latter does not is $P_i$. Thus, $U$ must contain $P_i$ as a part. Define $f(U)$ to be a partition that coincides with $U$ except it replaces $P_i$ with $Q_i,Q'_i$. Then $f(U) \in\dd_n\langle P\rangle\setminus \dd_n\langle P'\rangle$. Moreover, $f$ is an injection. Thus, $|\dd_n\langle P'\rangle|\le |\dd_n\langle P\rangle|$.

Repeatedly applying the splitting operation, we arrive at the all-singleton partition $S$, for which the number of derangements is thus minimized. Following \cite{KuRe}, we denote
$$\tilde{B_n}=|\dd_n\langle S\rangle|.$$
Note that $\dd_n\langle S\rangle$ is the family of all partitions with parts of size $\ge 2$.
We have the following recurrence relations:
\begin{align*}
  B_{n+1} =& \sum_{i=0}^n{n\choose i}B_i \\
  \tilde{B}_{n+1} =& \sum_{i=2}^{n-1}{n\choose i}\tilde{B_i}.
\end{align*}
Let us prove by induction on $s\ge 2$ that  we have
\begin{equation}\label{eqb3}\tilde B_s\ge \frac 12\prod_{i=2}^{s-1} \Big(1-\frac{2\log_e (i+1)}{i+1}\Big) \Big(1-\frac {2i+2}{3^{i}}\Big)\cdot B_s.\end{equation}
We have $B_{i+1}\ge 2B_{i}$ if $i\ge 1$, and thus $1=B_0=B_1$ and $B_i\ge 2^{i-1}$ for $i\ge 2$. Thus,
$$\frac{\sum_{i=2}^s {s\choose i}B_i}{\sum_{i=0}^s {s\choose i}B_i}\ge 1-\frac{s+1}{\frac 12+\frac 12\sum_{i=0}^s 2^i{s\choose i}}\ge 1-\frac{2s+2}{3^s}.$$
 Returning to \eqref{eqb3}, it is true for $s=2$ because $\tilde B_2=1$ and $B_2 = 2$.
We have (using induction in the first inequality below, the last displayed inequality in the third inequality below, and Lemma~\ref{lemb1} in the fourth inequality)
\begin{align*}\tilde B_{s+1} = \sum_{i=2}^{s-1}{s\choose i}\tilde{B_i}\ge& \frac 12\sum_{i=2}^{s-1}{s\choose i} \left(\prod_{j=2}^{i-1} \Big(1-\frac{2\log_e (j+1)}{j+1}\Big)\Big(1-\frac {2j+2}{3^{j}}\Big)\right) B_i\\
\ge& \frac 12 \left(\prod_{j=2}^{s-1} \Big(1-\frac{2\log_e (j+1)}{j+1}\Big)\Big(1-\frac {2j+2}{3^{j}}\Big)\right)\sum_{i=2}^{s-1}{s\choose i} B_i\\
\ge& \frac 12 \left(\prod_{j=2}^{s-1} \Big(1-\frac{2\log_e (j+1)}{j+1}\Big)\right)\left(\prod_{j=2}^{s} \Big(1-\frac {2j+2}{3^{j}}\Big)\right)\sum_{i=0}^{s-1}{s\choose i} B_i\\
   =& \frac 12\left(\prod_{j=2}^{s-1} \Big(1-\frac{2\log_e (j+1)}{j+1}\Big)\right)\left(\prod_{j=2}^{s} \Big(1-\frac {2j+2}{3^{j}}\Big)\right)(B_{s+1}-B_s)\\
   \ge& \frac 12\left(\prod_{j=2}^{s} \Big(1-\frac{2\log_e (j+1)}{j+1}\Big)\Big(1-\frac {2j+2}{3^{j}}\Big)\right) B_{s+1}.\end{align*}
This proves \eqref{eqb3}. Let us now obtain a bound on $\tilde B_n$. First, we note that $\prod_{i=2}^{s} \Big(1-\frac {2i+2}{3^{i}}\Big)\ge \frac 13\big(1-\sum_{i=3}^s \frac {2i+2}{3^{i}}\big)\ge \frac 13(1-\frac 8{27}-2\cdot \frac {10}{81})\ge \frac 19$. Next,
$$\prod_{i=3}^s\Big(1-\frac {2\log_e i}i\Big)\ge c e^{-\sum_{i=3}^s \frac {3\log_e i}i}\ge ce^{-\frac 32\log_e^2s}$$
for some positive constant $c$. In the first inequality we used the bound $(1-\frac 1x) = (1+\frac{1}{x-1})^{-1}\ge e^{-\frac 1{x-1}}$.

Thus, we get
$$\tilde B_n\ge \frac 1{18} c e^{-\frac 32\log_e^2n} B_n =c' e^{-\frac 32\log_e^2n} B_n,$$
for some positive constant $c'$.
\end{proof}

\begin{proof}[Proof of Theorem~\ref{thm6}] We give the following set interpretation to $\bb_n$. Consider the ground set $2^{[n]}$, and let $P\in \bb_n$ be mapped into a $\le n$-element set $A$ on $2^{[n]}$, where each element of $A$ corresponds to a part from the partition $P$. In what follows, we think of $\bb_n$ as a family of sets. Note that, using Lemma~\ref{lemb1}, for any set $X\in {2^{[n]}\choose s}$, $s\le n$, such that $\bb_n(X)$ is non-empty, we get
$$\Big(\frac{|\bb_n|}{|\bb_n(X)|}\Big)^{1/s}\ge \Big(\frac{B_n}{B_{n-s}}\Big)^{1/s}\ge \Big(\prod_{i=n-s}^{n-1}\frac{i}{2\log_e i}\Big)^{1/s}\ge \Big(\frac{(n-1)!}{(2\log_e (n-1))^{n-1}}\Big)^{1/(n-1)}\ge $$
$$ \frac{n/e}{2\log_e (n-1)}\ge \frac {n}{6 \log_e n},$$
where for the last inequality we need $n$ to be somewhat large (e.g., $n\ge 50$ is sufficent).

From the above, we get that the family $\bb_n$ is $r_0$-spread for $r_0=\frac n{6\log_e n}$. Moreover, it is weakly $(\frac n{12\log_e n},t)$-spread for any $t\le n/2$. Recall the definition of weakly $(r_0/2,t)$-spread family. We need to show that  $|\m B_n(X_0)|/|\m B_n(X)|\ge (r_0/2)^{|X|-|X_0|}$ for any $X$ of size greater than or equal to  $t$ and for the choice of $X_0$ such that $|\m B_n(X_0)|$ is the largest.
Take $\m B_n(X_0)$ to be to all partitions that fix $t$ prescribed singletons. We have $|\m B_n(X_0)| = B_{n-t}$. We need to compare that with the number of partitions that fix some $t+s$ parts, $s\ge 0$, and there are at most $B_{n-t-s}$ of those for any choice of the fixed parts. To bound the ratio, we reuse the bounds displayed above.

We are now ready to prove the theorem. Take a $t$-intersecting family $\ff\subset\bb_n$. First, we apply Theorem~\ref{thmkz1} with $r = r_0/2$ and $q = 2^{-10}\frac{n}{\log_2 n}$. We get a $t$-intersecting family $\s$
of sets of size at most $q$ (corresponding to collections of $q$ pairwise disjoint sets) and a remainder family $\ff'$ such that $$|\ff'|\le 2^{-q-1}B_n\le n^{-q/\log_2 n} B_n\le n^{-t-4\log_en }B_n\le n^{-4\log_en}B_{n-t}.$$ Here we used a bound $B_n/B_{n-1}\le n$, which is valid since any partition of $n-1$ elements can be prolonged in at most $n$ ways to a partition on $n$ elements, as well as the bound $t+4\log_e n\le \frac q{\log_2 n}$, valid for our choice of parameters.

Next, we apply Theorem~\ref{thmkz2} to the family $\s$ with $\epsilon = 1/2$. The family $\bb_n$ is weakly $(r_0/2,t)$-spread, and we have $r_0/2\ge 48q$, and thus the inequality on the parameters is satisfied. We conclude that either $|\ff|\le \frac 12 B_{n-t}+|\ff'|\le 0.6 B_{n-t}$, or $\s$ consists of one $t$-element set $S$. Moreover, in order for $\ff$ to be maximal, this set in the partition language must clearly correspond to a collection of $t$ singletons. Otherwise, $|\ff|\le B_{n-t-1}+|\ff'|\le 0.6 B_{n-t}$ again.

Finally, let us show that $\ff'\setminus \bb_n[S]$ must be empty.  Indeed, assume $P\in \ff'\setminus \bb_n[S]$. Then, by Lemma~\ref{lemb2}, the number of partitions in $\bb_n(S)$ that are derangements with respect to $P$ (induced on the complement of $S$ in the partition language), is at least $c' e^{-\frac 32\log_e^2 (n-t)} B_{n-t}$, which is larger than $|\ff'|$ by the last displayed formula. We get that if $\ff'$ is non-empty, then $\ff$ cannot be extremal. This completes the proof of the theorem.
 \end{proof}

\subsection{Partitions with $k$ parts. Proof of Theorem~\ref{thmp1}}
We will need the following relation between Stirling coefficients of the second kind.
\begin{lem}\label{lemp1}
  For each $n\ge 1+2\ell\log_2 n$, $\ell\ge 2$ we have $\stir n\ell\ge n^2 \stir{n-1}{\ell-1}$.
\end{lem}
\begin{proof}
  Stirling numbers obey the following recurrence relation:
  $$\stir{n}{\ell}=\stir{n-1}{\ell-1}+\ell\stir{n-1}{\ell}.$$
  From here, we see that $\stir{n}{\ell}\ge \ell\stir{n-1}{\ell}$. Next, let us compare $\stir{n-1}{\ell}$ and $\stir{n-1}{\ell-1}$. Consider a bipartite graph between $\ell$-partitions and $(\ell-1)$-partitions of $[n-1]$, where two partitions are connected by an edge if one is a refinement of the other. Let us count the degrees in this graph. The neighbors of an $\ell$-partition $X$ are the $(\ell-1)$-partitions obtained by merging two parts of $X$. Thus, the degree of $X$ is ${\ell\choose 2}$. The neighbors of an $(\ell-1)$-partition $Y$ are those that can be obtained by subdividing one of the parts of $Y$ into non-empty parts. Assume that the parts $Y_1,\ldots, Y_\ell$ have sizes $k_1,\ldots, k_\ell$, $\sum_{i=1}^\ell k_i = n$. There are $2^{\ell_i-1}-1$ ways to subdivide $Y_i$. Thus, the degree of $Y$ is $\sum_{i=1}^\ell(2^{k_i-1}-1)\ge \ell 2^{-1+(n-1)/\ell}-\ell$ by convexity. The last expression is at least $\ell n^2/2-\ell$ by our assumption on $n$. Concluding, the degree of any element in the $\ell$-partitions part is $\ge n^2/\ell$ times smaller than the degree of any element in the $(\ell-1)$-partitions part. Double counting the number of edges, we get that $\stir{n-1}{\ell}\ge n^2/\ell\stir{n-1}{\ell-1}$ in our assumptions. Combining it with the displayed formula, we get that $\stir{n}{\ell}\ge n^2\stir{n-1}{\ell-1}$.
\end{proof}

\begin{proof}[Proof of Theorem~\ref{thmp1}]  We interpret  $\pp_n^\ell$ as a subfamily of $\bb_n$.  That is, the ground set is $2^{[n]}$, and each $X\in \pp_n^\ell$ is mapped into an $\ell$-element set on $2^{[n]}$, where each element of the set corresponds to a part from the partition $X$. In what follows, we think of $\pp_n^\ell$ as a family of $\ell$-element sets.

Let us derive that $\pp_n^\ell$ is weakly $(\frac {n^2} 2,t)$-spread for each $t\le \ell-2$. We take a collection of $t$ distinct singletons as the set $T$ in the definition of a weakly spread partition. The number of $\ell$-partitions extending it is $\stir{n-t}{\ell-t}$. The number of partitions with any $t+s$ fixed parts is at most $\stir{n-t-s}{\ell-t-s}$. Using Lemma~\ref{lemp1}, we have
$$\stir{n-t}{\ell-t}/\stir{n-t-s}{\ell-t-s}\ge (n-t-s)^{2s}\ge (n^2/2)^s,$$
provided $\ell-t-s\ge 1$. If $\ell-t-s=0$ then we combine it with the bound $\stir{n-t-s+2}2 =2^{n-t-s+1}-1\ge 2^{n/2} \ge n^4=n^4\stir{n-t-s}{\ell-t-s}$ (the last inequality is valid due to our choice of $n$). Thus, the last displayed inequality is always true, and we conclude that $\pp_n^\ell$ is weakly $(\frac {n^2}2,t)$-spread for each $t\le \ell-2$.

Apply Theorem~\ref{thmkz2} to $\ff$ with $\varepsilon=1/2$, $\ell$ playing the role of $q$ and $r= n^2/2$. We need to check if $r\ge 48\ell$. The inequality clearly holds for $n$ satisfying $n\ge 2\ell\log_2 n$ and $n\ge 48$. We conclude that either $\ff$ is a family of partitions extending a fixed set of $t$ singletons, or $|\ff|\le \frac 12\stir{n-t}{\ell-t}$. This concludes the proof of the theorem.
\end{proof}

\subsection{Profiled $t$-intersecting partitions. Proof of Theorem~\ref{thmu1}}
 For this proof, we also interpret $\uu_P$ as a subfamily of the family corresponding to $\bb_n$. This way, we have $\uu_P\subset {2^{[n]}\choose \ell}$.
In this setup, we also can directly apply Theorem~\ref{thmkz2} with $\epsilon = 1/2$ and get the desired conclusion, provided that we can show that $\uu_{k,\ell}$ is a weakly $(r,t)$-spread family with $r\ge 48\ell$. We will show this below.  Let $a_t$ be the number of partitions that contain fixed parts of sizes $k_1,\ldots, k_t$, and let $b_U$ be the number of partitions that contain fixed parts of sizes $k_i, i\in U$. Note that, for any $w$ and $U$ with $|U|=w$ we have $a_w\ge b_U$. Let $U$ be a subset of $[\ell]$ of size $t+s$. For shorthand, let us denote $n_j=\sum_{i=j+1}^\ell k_i$. 
In order to count the numbers of partitions, we need to split $k_{t+1},\ldots, k_{\ell}$ into groups by their size. Assume that $g_0:=t< g_1<g_2\ldots <g_m:=\ell$ are defined as follows: $k_{t+1} =\ldots = k_{g_1}<k_{g_1+1}=\ldots = k_{g_2}<\ldots$. Thus, among $k_{t+1},\ldots, k_{s}$ there are $m$ distinct values, specifically, for each $i\in [m]$ there are $g_{i}-g_{i-1}$ parts of size $k_{g_i}$. Also, take $j$ such that $g_{j-1}<t+s+1\le g_{j}.$ Then we have
\begin{align*}\frac {a_t}{b_U}\ge \frac {a_t}{a_{t+s}}&= \frac{\frac{n_t!}{\big(\prod_{i=1}^m(g_i-g_{i-1})!\big)\big(\prod_{i=t+1}^\ell k_i!\big)}}{\frac{n_{t+s}!}{\big((g_j-(t+s))!\prod_{i=j+1}^m(g_i-g_{i-1})!\big)\big(\prod_{i=t+s+1}^\ell k_i!\big)}}\\
&=\frac{(g_j-(t+s))!n_t!}{\big(\prod_{i=1}^j(g_i-g_{i-1})!\big)n_{t+s}!\prod_{i=t+1}^{t+s} k_i!}\\
&\ge \frac{(g_j-(t+s))!n_t!}{(g_j-t)! (n_{t}-3s)!\prod_{i=t+1}^{t+s} 3!}\\
&\ge \frac{n_t!}{(\ell-t)^s (n_{t}-3s)!6^s}
\end{align*}

Let us comment on the first inequality. We bounded the first product in the denominator using $(a+b)!\ge a!b!$. We also used  that for $k\ge k'$ we have $k!k'!\le (k+1)!(k-1)!$, as well as that $k_i\ge 3$ for $i\ge t+1$. We iteratively apply it to show that, in the conditions $k_i\ge 3$ and $n_{t+s}+\sum_{i=t+1}^{t+s}k_i = n_t$, the expression $n_{t+s}!\prod_{i=t+1}^{t+s} k_i!$ is maximized when the first term is $(n_t-3s)$, and all the $k_i$'s are equal to $3$.

Below, we use that $(n_t)!/(n_t-3s)!\ge ((n_t)!)^{(3s)/n_t}\ge (n_t/e)^{3s}$ and that $n_t\ge 3(\ell-t)\ge \frac 32 \ell$, since $\ell\ge 2t$.  The last displayed expression is at least
$$\Big(\frac{n_t^3}{e^3 6(\ell-t)}\Big)^s\ge \Big(\frac{\ell^2}{20}\Big)^s.$$
We get that the family $\uu_P$ is weakly $(\ell^2/10,t)$-spread. Since $\ell^2/20> 48\ell$ for $\ell\ge 1000$, this is sufficient for our application of Theorem~\ref{thmkz2}.


\subsection{Partially $t$-intersecting partitions. Proof of Theorem~\ref{thm1}}
In what follows, we assume that $k\ge 3$. First, we interpret $\uu_{k,\ell}$ as a family of sets. Put $N:= {k\ell\choose 2}$ and correspond to each partition $P$ in $\uu_{k,\ell}$ the set of all pairs of elements $x_1,x_2\in [k\ell]$ such that $x_1$ and $x_2$ belong to the same part in $P$. Note that, in graph terms, $P$ is a collection of $\ell$ $k$-cliques. As a set, each $P$ has size $\ell {k\choose 2}$. In what follows, we often work with partitions in this set form and think of $\uu_{k,\ell}$ as a family in $\binom{[N]}{{k\choose 2}\ell}$. 

We shall show that, as a set family, $\uu_{k,\ell}$ is sufficiently spread, provided $\ell$ is large enough. For that, we need to analyze the numbers of different partitions. Consider a collection $X = \{X_1,\ldots, X_a\}$ of disjoint sets in $[k\ell]$, such that $2\le |X_i|\le k$. We call $X$ {\it a subpartition}. Put $m(X) = \sum_{i=1}^a (|X_i|-1)$. We say that a $(k,\ell)$-partition $P = (P_1,\ldots, P_\ell)$  {\it extends} a subpartition $X$ if for each $j$ there is $i$ such that $X_j\subset P_i$, and that $P$ {\it extends $X$ in an injective way} if these $i$'s are distinct for distinct $j$'s. The difference is that in the first case $P_i$ may simultaneously contain several parts from $X$.  
We shall use the following inequality for $a>b$: $a!/b!\ge (a!)^{(a-b)/a}\ge (a/e)^{a-b}.$
We claim that the number $\tilde u_{k,\ell}(X)$ of partitions that extend $X$ in an injective way is at most $\big(\frac 9\ell \big)^{m(X)} u_{k,\ell}.$ Indeed, we have the following bound:
$$\frac{\tilde u_{k,\ell}(X)}{u_{k,\ell}}=\frac{\frac{(k\ell-m(X)-a)!} {(\ell-a)!(k!)^{\ell-a}\prod_{i=1}^a (k-|X_i|)!}}{\frac{(k\ell)!}{\ell!(k!)^\ell}}\le$$$$ \frac{k^{m(X)+a}\ell^a}{(k\ell/e)^{m(X)+a}}\le  \frac {e^{m(X)+a}}{\ell^{m(X)}}\le \frac {9^{m(X)}}{\ell^{m(X)}},$$
where in the last inequality we used that $m(X)\ge a$. Note that this holds for subpartitions $X$ that fix 
at least $2$ elements in each ``active'' block. 

The value that we shall need in what follows is  the number $u_{k,\ell}(X)$ of partitions that extend $X$ (but not necessarily in an injective way). Any such partition extends $X'$ in an injective way, where $X'$ is obtained from $X$ by merging some parts. Assume that $X'$ contains $r$ parts that were obtained by merging of some of the parts of $X$, and that the total number of parts that participated in merging is $q$. Then $X'$  contains $a-q+r$ parts, where $q-r\ge q/2$, and we have  $m(X') = m(X)+q-r$. For each such $X'$ we have $\tilde u_{k,\ell}(X')/u_{k,\ell}\le (9/\ell)^{m(X)+q-r}$. Next, let us bound the number of possible $X'$ for given values $q,r$. First, we select $q$ parts of $X$ out of $a$ to be merged, and then for each part we decide, which of the $r$ merged parts it would go to. (At this point, we ignore the fact that each merged part must include at least two original parts, so this is clearly an overestimate.)  Thus we could bound the number of such $X'$ for given values of $q,r$ as follows: ${a\choose q} r^{q}\le \big(\frac{aer}{q}\big)^q\le (2a)^q$, since $r\le q/2$. 
Then we can bound $u_{k,\ell}(X)/u_{k,\ell}$ as follows (in the chain of inequalities, we bound a sum by the sum of a geometric progression twice, using that $\ell$ is sufficiently large)
\begin{align}\notag\frac{u_{k,\ell}(X)}{u_{k,\ell}}\le \sum_{X'}\frac{\tilde u_{k,\ell}(X')} {u_{k,\ell}}&\le \sum_{q=2}^a\sum_{r=1}^{q/2}\frac{(2a)^{q} 9^{m(x)+q-r} }{\ell^{m(x)+q-r}}\le 2 \sum_{q=2}^a\frac{(2a)^{q} 9^{m(x)+q/2} }{\ell^{m(x)+q/2}}\\
\label{eqpart3}\le 2 \sum_{q=2}^a\frac{(k\ell)^{q} 9^{m(x)+q/2} }{\ell^{m(x)+q/2}}&=2 \sum_{q=2}^a\frac{(9k^2 \ell)^{q/2} 9^{m(x)} }{\ell^{m(x)}}\le 4 \frac{(9k^2 \ell)^{a/2} 9^{m(x)} }{\ell^{m(x)}}\le \frac{(9^{3}k)^{m(x)} }{\ell^{m(x)/2}},
\end{align}
where we used that $a\le k\ell/2$ when passing to the second line (the ground set has size $k\ell$, thus we cannot have more than $k\ell/2$ parts of size at least $2$), and in the last inequality we used $a\le m(x)$ again.

The family of partitions that extend $X$ can be expressed in different ways in set terms.\footnote{In what follows, we refer to elements of $[N]$ as edges in the graph sense. It is natural since they correspond to pairs of elements of the original ground set.} Actually, it is necessary and sufficient for each $i\in [a]$ to fix a subgraph on $X_i$ that is connected, and add no other edges. Thus, the largest number of edges we can fix is $\sum_{i=1}^a{|X_i|\choose 2} = \sum_{i=1}^a \frac {|X_i|(|X_i|-1)}2\le \frac k2 \sum_{i=1}^a(|X_i|-1) = \frac {km(X)}2$. Put differently, if we have fixed $x$ edges, then  the corresponding $X$ satisfies $m(X)\ge\lceil\frac {2x}k\rceil$. Moreover, the partitions that we get are precisely the partitions that extend $X$, where $X$ is the set of connected components of size at least $2$ defined by these $x$ edges. Thus the bound \eqref{eqpart3} is valid for all subpartitions that may arise this way.
From here, we get that the family $\uu_{k,\ell}$ is $(\frac \ell {9^6k^2})^{1/k}$--spread.

If $\ell$ is sufficiently large (say, $\ell>k^{Ck}$ with some large constant $C$), this spreadness is sufficient to apply the spread approximation machinery.
But before we do so, let us analyze, what happens with the $t$-intersecting property when passing to the set interpretation. If two $(k,\ell)$-partitions partially $t$-intersect, then they have ${t\choose 2}$ edges in common. In what follows, we will be working with this ${t\choose 2}$-intersection property for set families. However, there is a complication that we have to overcome: there are obviously many other ways for two sets to have intersection of size $\ge {t\choose 2}$, without the corresponding partitions necessary being partially $t$-intersecting. Luckily, being partially $t$-intersecting is the most ``economical'' way, which allows us to overcome this complication.

Consider a ${t\choose 2}$-intersecting (in the set sense) family $\ff\subset \uu_{k,\ell}$. Apply Theorem~\ref{thmkz1} with $r = \big(\frac \ell{9^6k^2}\big)^{1/2k}$, $r_0 = \big(\frac \ell {9^6k^2}\big)^{1/k}$, $q = k^{10}$. The uniformity $\ell {k\choose 2}$ plays the role of $k$ from the theorem. We need $\big(\frac \ell{9^6k^2}\big)^{1/2k}>2k^{10}$ and $\big(\frac \ell{9^6k^2}\big)^{1/2k}>2^{12}\log_2(2\ell{k\choose 2}).$ Both inequalities on $r$ are valid, provided, say, $\ell> k^{100k}$.
This allows us to apply Theorem~\ref{thmkz1} and get a family $\s$ of sets of size at most $q$ that cover most partitions in $\ff$, and a remainder family $\ff'\subset \ff$ satisfying
\begin{equation}\label{eq24}|\ff'|\le (r/r_0)^{q+1}u_{k,\ell} \le \Big(\Big(\frac {9^6k^2}\ell\Big)^{1/2k}\Big)^{k^{10}}u_{k,\ell}\le \ell^{-k^{7}}u_{k,\ell}.\end{equation}

The next step is to apply Theorem~\ref{thmkz2} to our approximation $\s$ with ${t\choose 2}$ playing the role of $t$. Before we do so, we need to show that $\uu_{k,\ell}$ possesses the weak $(r',{t\choose 2})$-spreadness property with a sufficiently large $r'$. In order to do so, we need to return to the analysis of the subpartitions. Let $E$ be a collection of ${t\choose 2}$ edges and let $X = (X_1,\ldots, X_a)$ be the corresponding subpartition. Then ${t\choose 2} = \sum_{i=1}^a {|X_i|\choose 2}$. Recall that $m(X) = \sum_{i=1}^a(|X_i|-1)$.

\begin{lem} Suppose that ${t\choose 2} = \sum_{i=1}^a {|X_i|\choose 2}$.
\begin{itemize}
    \item[(i)] We have $m(X)\ge t-1$, moreover, $m(X)=t-1$ if and only if $a=1$, and the subpartition consists of just $1$ set $X_1$ of size $t$.
  \item[(ii)] If $E$ is a collection of ${t\choose 2}+s$ edges, then the corresponding partition $X(E)$ satisfies $m(X(E))\ge t-1+\frac sk$.
\end{itemize}
\end{lem}
\begin{proof}
(i) If $\max |X_i|>t$ then $m(X)>t-1$ and we are done. Thus, assume that $\max|X_i|\le t$ and $m(X)\le t-1$. We have $$\sum_{i=1}^a {|X_i|\choose 2}\le \max |X_i|\cdot\frac{m(X)}2\le t\cdot \frac{t-1}2 = {t\choose 2},$$ where the last inequality turns into equality iff $\max |X_i|=t$.

(ii) Let $i\in [t,k]$ be such that either $i=k$ and $|E|\ge {k\choose 2},$ or $i\le k-1$ and ${i\choose 2}\le |E|<{i+1\choose 2}.$ Let us deal with the easier case of $i=k$. Note that in any case we have $|X_i|\le k$. We have
\begin{equation}\label{eqsh1} |E|\le\sum_{i=1}^a {|X_i|\choose 2}\le \max |X_i|\cdot\frac{m(X)}2\le k\cdot \frac{m(X)}2,\end{equation}
and thus $m(X)\ge \frac {2|E|}k$. Let $|E|={k\choose 2}+\ell= {t\choose 2}+s$, and so $s= {k\choose 2}-{t\choose 2}+\ell$. Note that ${k\choose 2}-{t\choose 2} = (k-1)+(k-2)+\ldots+t\le (k-t)(k-1).$ We have
$$m(X)\ge \frac {2|E|}k=k-1 +\frac{2\ell}k\ge t-1+\frac{(k-t)(k-1)}{k}+\frac{\ell}k\ge t-1+\frac{{k\choose 2}-{t\choose 2}+\ell}{k}=t-1+\frac sk.$$
Next, assume that $i\le k-1.$  If $\max|X_i|\le i$, then we can lower bound $m(X)$ as in \eqref{eqsh1}: $m(X)\ge \frac{2|E|}i.$ Let us put $|E|={i\choose 2}+\ell$, where $0\le \ell<i.$ We will use the following inequality: ${i\choose 2}-{t\choose 2} =(i-1)+(i-2)\ldots+t\le (i-t)(i-1)$. We have
$$m(X)\ge \frac {2|E|}i=i-1 +\frac{2\ell}i\ge t-1+\frac{(i-t)(i-1)}{i}+\frac{\ell}i\ge t-1+\frac{{i\choose 2}-{t\choose 2}+\ell}{i}\ge t-1+\frac sk.$$

If $|X_i|\ge i+1$ for some $i\le k-1$ then we have
$$m(X)\ge i=t-1+\frac{(i-t+1)i}i\ge t-1+\frac{{i+1\choose 2}-{t\choose 2}}i\ge t-1+\frac sk.$$
This completes the proof.\end{proof}

Choose a collection $T$ of ${t\choose 2}$ edges  that corresponds to a subpartition of $1$ $t$-element set (in other words, a $t$-clique). Once we are equipped with the property that for any collection of edges $E$ of size ${t\choose 2}+s$ the corresponding partition $X$ satisfies $m(X)\ge t-1+\frac sk$, it is easy to do similar calculations as for $u_{k,\ell}$ and $u_{k,\ell}(X)$ above, and obtain that $\uu_{n,k}(T)$ is $r'$-spread with $r' = \big(\frac \ell{9^6k^2}\big)^{1/2k}$. We note that this weak $(r,t')$-spreadness property is subtle in this application: the family $\uu_{k,\ell}$ only possesses it for $t' = {t\choose 2}$ for integer $t$.

The next step is to apply Theorem~\ref{thmkz2} to our approximation $\s$. Let  ${t\choose 2}$ play the role of $t$, $\epsilon= 1/2$, $r'$ playing the role of $r$, and $q=k^{10},$ as above. Again, we can see that the inequality on $r'$ is valid, provided $\ell$ is large enough (again $\ell>k^{Ck}$ is sufficient). The conclusion is that $\s$ must consist of a single set $T$ of size ${t\choose 2}$, otherwise the size of $\uu_{k,\ell}[S]$ is at least twice smaller than $\uu_{k,\ell}[T]$ for the largest $T$. Moreover, $T$ must correspond to a subpartition consisting of one $t$-element set.

At this point, we have proved a rough version of the conjecture, along with stability: if the size of a partially $t$-intersecting family of $(k,\ell)$-partitions is at least $0.51|\cc^T_{k,\ell}|$, then it is contained in some $\cc^T_{k,\ell}$, with an exception of at most $\ell^{-k^7}u_{k,\ell}$ sets. These exceptional sets form the family $\ff'$ from above, and we next aim to show that this family is empty. (In what follows, we assume that $\ff'\cap \cc^T_{k,\ell}=\emptyset$, since otherwise we could move these sets to $\ff\setminus \ff'$.)

In order to show that $\ff' = \emptyset$ for an extremal $\ff$, we need to get some understanding on how often do random $(k,\ell)$-partitions partially $t$-intersect.
\begin{lem}\label{lemest1} Let $t\ge 2$.
  For a given set $T$ of size $t$ and a $(k,\ell)$-partition $Y = (Y_1,\ldots, Y_\ell) \notin \cc^T_{k,\ell}$, the number of other $(k,\ell)$-partitions from $\cc^T_{k,\ell}$ that do not partially $t$-intersect it is at least $\ell^{-2k^2}u_{k,\ell}$.
\end{lem}
\begin{proof}
  Fix some partition $X = (X_1,\ldots, X_\ell) \in \cc^T_{k,\ell}$ and consider a random permutation $\sigma: [k\ell]\setminus T\to [k\ell]\setminus T$. We will prove that  $$\Pr[\sigma(X) \text{ partially $t$-intersects } Y]\le 1-\ell^{-k^2}.$$ Given that $|\cc^T_{k,\ell}|\ge \ell^{-k^2}u_{k,\ell}$ and that a uniform random permutation $\sigma$ generates a uniformly random element of $\cc^T_{k,\ell}$, the statement of the lemma follows from the displayed formula.

  We shall expose $\sigma$ block by block, where the $i$-th block describes where  $X_i$ is mapped. Let us denote by $A_i$ the  event that $|\sigma(X_i)\cap Y_j|\ge t$ for some $j$. We suppose that $T\subset X_1$. Let us first deal with the  ``exceptional'' first event $A_1$. There are two possible ways for it to occur. First, it is possible that for some $j$ we have $|Y_j\cap T|=t-1$ and $\sigma(X_1\setminus T)\cap Y_j\ne \emptyset$. Note, that there are at most $2$ such $j$ (at most $1$ for $t\ge 3$). To bound this probability, we simply look at the probability that these sets intersect. The second possibility is covered by the event that $|\sigma(X_1\setminus T)\cap Y_j|\ge 2$ for some $j$. This is the event  that the sets of edges of $X_1\setminus T$ and of the partition $Y$ intersect. Thus, we can provide the following simple bound:
  \begin{equation}\label{eq20}\Pr[\sigma(X_1)\cap Y_j\ge t \text{ for some }j]\le \frac{2k|X_1\setminus T|}{k\ell-t}+ \frac{{k\choose 2}\ell{|X_1\setminus T|\choose 2}}{{k\ell-t\choose 2}}\le \frac 12.\end{equation}
  Next, we are going to bound the probability that $A_i$ happens, given that none of the $A_1,\ldots, A_{i-1}$ happened. Each of these events is included in the event that  $|\sigma(X_i)\cap Y_j|\ge 2$ for some $j$. This is the event that  the sets of edges of $X_1\setminus T$ and of the partition $Y$ intersect. At this point, we are working with the partition that $Y$ induces on $[k\ell]\setminus \sigma(X_1\cup\ldots\cup X_{i-1})$. The latter set has size $k(\ell-i+1)$. Restricted to it, $Y$ is a partition into parts of size at most $k$ each. Thus, the number of edges in $Y$ is at most ${k\choose 2}(\ell-i+1)$, and we have the bound
  \begin{equation}\label{eq21}\Pr[A_i |\bar A_1,\ldots, \bar A_{i-1}]\le \frac{{k\choose 2}\sum_{j=1}^\ell{|Y_j\setminus (\sigma(X_1\cup\ldots\cup X_{i-1}))|\choose 2}}{{k(\ell-i+1)\choose 2}}\le \frac{{k\choose 2}^2 (\ell-i+1)}{{k(\ell-i+1)\choose 2}}\le \frac{k^2}{2(\ell-i)}.\end{equation}
  We will use this bound up to $i = \ell-k^2$. The remaining parts, that is, $R:=\sigma(X_{\ell-k^2+1}\cup\ldots\cup X_{\ell})$ form a set of size $k^3$. The partition $Y$ induced on $R$ again consists of sets of size at most $k$, and clearly there is at least $1$ choice of $\sigma$ so that each $X_i$ does not intersect each part of $Y$ induced on $R$ in more than $1$ element. At the same time, the number of different partitions of $R$ into $k^2$ parts of size $k$ is at most $k^{3k^3}$, and thus
  \begin{equation}\label{eq22}\Pr[\bar A_{\ell-k^2+1}\cap\ldots\cap \bar A_{\ell}\ | \ \bar A_1,\ldots, \bar A_{\ell-k^2}]\ge k^{-3k^3}.\end{equation}
  Combining the bounds \eqref{eq20}, \eqref{eq21}, \eqref{eq22}, we get
  $$\Pr[\cap_{i=1}^\ell\bar A_i]\ge \frac 12k^{-3k^3}\prod_{i=1}^{\ell-k^2} \Big(1+\frac{k^2}{2(\ell-i)-k^2}\Big)^{-1}\ge \frac 12k^{-3k^3}e^{-\sum_{i=1}^{\ell-k^2}\frac{k^2}{2(\ell-i)-k^2}}\ge $$
  $$\frac 12k^{-3k^3}e^{-\frac {k^2}2\log_e \ell} = \frac 12 k^{-3k^3}\ell^{-\frac {k^2}2}\ge \ell^{-k^2}. $$
  This completes the proof of the lemma.
\end{proof}

Returning to the extremal family $\ff$, assume that $\ff'$ is non-empty and thus contains some partition $Y$. Lemma~\ref{lemest1} implies that $|\cc^T_{k,\ell}\setminus \ff|\ge \ell^{-2k^2}u_{k,\ell}$. But \eqref{eq24} implies that $|\ff'|\le \ell^{-k^7}u_{k,\ell}$. We conclude that $|\ff'|\ll |\cc^T_{k,\ell}\setminus \ff|$, and thus $\ff'$ cannot be extremal unless $\ff'$ is empty. This concludes the proof of the theorem.

\section{Acknowledgements}
We are indebted to the referees for carefully reading the text and pointing out several problems with the presentation, as well as several mistakes. The work was partially supported by the Foundation for the Advancement of Theoretical Physics and Mathematics “BASIS” and by  the Ministry of Science and Higher Education of the Russian Federation, project No. FSMG-2024-0025.

\end{document}